\documentclass[a4paper,12pt]{article}
\usepackage{amssymb,amsmath}
\usepackage[dvips]{color}
\usepackage[dvips]{graphics}
\usepackage{algorithmic}
\usepackage{algorithm}

\newcommand{\normt}[1]{\|#1\|}
\newtheorem{theorem}{Theorem}[section]
\newtheorem{lemma}{Lemma}[section]
\newtheorem{remark}{Remark}[section]
\newtheorem{example}{Example}[section]
\newtheorem{definition}{Definition}[section]

\newcommand{\normal}[1]{#1^T\!~#1}

\newcommand{\beeq}{\begin{equation}}
\newcommand{\eneq}{\end{equation}}
\newcommand{\mbf}{\mathbf}     

\newcommand{\macheps}{\varepsilon_M}

\begin{document}
\title{Numerical solution of  saddle point problems by block  {Gram--Schmidt} orthogonalization}
\author{ Felicja Okulicka-D{\l}u\.{z}ewska and Alicja Smoktunowicz\thanks{Faculty of Mathematics and Information Science, Warsaw
University of Technology, Koszykowa 75, Warsaw, 00-662
Poland, e-mails:  F.Okulicka@mini.pw.edu.pl, A.Smoktunowicz@mini.pw.edu.pl}}
\maketitle

\begin{abstract}
Saddle point problems arise in many important practical applications. 
In this paper we propose and analyze some algorithms for solving symmetric saddle point problems 
which are based upon the block Gram-Schmidt method. 
In particular, we prove that the algorithm BCGS2 (Reorthogonalized Block Classical Gram-Schmidt) using Householder Q-R decomposition 
implemented in floating point arithmetic is backward stable, under a mild assumption on the matrix $M$. 
This means that the computed vector $\tilde z$  is the exact solution to a slightly perturbed linear system of equations $Mz = f$. 
\end{abstract}

\medskip


\medskip

\noindent {\large\bf AMS Subj. Classification:} 15A12, 15A23, 15A60, 65H10.

\section{Introduction}

We consider a symmetric saddle point problem
\begin{equation}\label{M}
M z = f               \Leftrightarrow
\left(
\begin{array}{cc}
A & B \\
B^{T} & -C \\
\end{array}
\right) \,\,
\left(
\begin{array}{c}
x \\
y \\
\end{array}
\right) \,\,
=
\left(
\begin{array}{c}
b \\
c \\
\end{array}
\right),
\end{equation}
where $A \in  \mathbb{R}^{m\times m}$ is symmetric positive definite ($A>0$),   $C \in  \mathbb{R}^{n\times n}$ is symmetric semipositive definite ($C \ge 0$),     $B \in  \mathbb{R}^{m\times n}$ has full column rank, $n=$rank$(B) \leq m$.
Then $M$ is nonsingular and there exists  a unique solution $(x_{*},y_{*})$ of (\ref{M}) (i.e. $M z_{*}=f$) .

These problems feature  in many contexts; for example,  in nonlinearly constrained optimization, structural mechanics,
 computational fluid dynamics, elasticity problems,
  mixed (FE) formulations of II and IV order elliptic PDE's,  
  weighted least squares (image restoration),   FE formulations of
coupled  problem  (see \cite{Benzi: 05}). 
Coupled  problems are common  in the real world; description of
different mechanical phenomena,  such as flow and thermal effects, leads to  coupled systems of differential
equations.  The finite element method (FEM) is widely used to solve such
problems, and the most important part of the finite element method
algorithm is the procedure for  solving the set of linear equations possesing  saddle point structure
(see  \cite{Dlu,Oku1,Oku2}). 

This problem structure is naturally suited to the application of   block algorithms, and  block algorithms are suitable for parallel implementation, as they 
allow the splitting of  data and computation onto separate memories and computation devices.
Block methods  operate on groups of columns of $M$ instead of columns,   
to create a BLAS--3  compatible algorithm, that is, 
an algorithm built upon matrix--matrix operations.

It is known (see \cite{Benzi: 05}) that the block LU factorization of $M$ is in general unstable, and so in this paper we  study the numerical properties of  block Q-R orthogonalization for the solving of (\ref{M}). 
This approach may be a satisfactory  alternative to iterative methods,  where  often the number of iterative steps is unknown,  and further it 
 is not always possible to find a proper preconditioner. 
For a recent account of the theory we refer the reader to \cite{Benzi: 05},   \cite{Arioli: 00} and  \cite{GolGreif}.
 In contrast to  iterative methods, we  prove that the proposed method  is  numerically stable,  under a mild assumption on the matrix $M$.

The paper is organized as follows.  Section $2$ describes  two block Q-R decompositions of the matrix $M$,   BCGS (Block {Gram-Schmidt}) and BCGS2     (Reorthogonalized Block {Gram-Schmidt}).
Section $3$ examines the  numerical stability of these methods when  used  for solving the system (\ref{M}).  Section $4$ is devoted to numerical experiments and comparisons of the methods.

Throughout the paper, $\normt{\cdot}$ denotes the matrix or vector two--norm depending upon context, and $\kappa(M)=  \normt{M^{-1}}     \normt{M} $ is  the standard condition number of $M$.

\section{Algorithms}

We derive  error bounds for the solution of the system (\ref{M}) using the block Q-R methods BCGS (Block Classical {Gram-Schmidt}) 
and BCGS2  (Block Classical {Gram-Schmidt with reorthogonalization}). 
These algorithms produce a Q-R factorization of the matrix $M \in \mathbf{R}^{l \times l}$, where $l=m+n$. 
Then $M=QR$, where  $Q =(Q_1,Q_2) \in   \mathbf{R}^{l \times l}$ is orthogonal ($Q^TQ=I$) 
and $R \in   \mathbf{R}^{l \times l}$ is upper triangular matrix with positive diagonal entries.

Given the block Q-R decomposition of $M$, the solution $z$ to a system (\ref{M}) can be obtained by computing 
first the vector $g=Q^T f$   and then by
 solving the system $R z=g$ with upper triangular matrix $R$ by back-substitution.

At the core of the orthogonalization methods  BCGS and BCGS2  is a  orthogonal factorization routine 
the thin Householder, which for a matrix $X \in \mathbb{R}^{l\times k}$, 
where $l \ge k=rank(X)$,  produces a left orthogonal matrix  $Q  \in \mathbb{R}^{l\times k}$ ($Q^TQ=I$),    
and upper triangular matrix $R  \in \mathbb{R}^{k\times k}$ such that $X=QR$.
In  MATLAB  we use  the statement $[Q,R]=\mbf{qr}(X,0)$. It is well known (see, e.g., \cite{Golub96},  \cite{Bjor96})  that Householder Q-R is 
{\bf unconditionally stable}. In floating-point arithmetic with machine precision $\macheps$,  the thin Householder Q-R produces
the factors $\tilde Q$ and $\tilde R$ such that
\begin{eqnarray}
\normt{ I -\normal{\tilde {Q}} }  &\leq& \macheps \, \, L ,  \label{house1} \\
 X +  \Delta X &=&\tilde {Q} \tilde {R},\quad \normt{\Delta X} \leq  \macheps \,\, L   \normt{X}  \label{house2}
\end{eqnarray}
for some modest constant $L=L(l,k)$. 
Instead of the thin Householder Q-R we can use other stable Q-R factorizations,  for example Givens Q-R (see \cite{Bjor96}, \cite{Golub96}). 

\medskip

\begin{algorithm}[H]
\caption{ Block Classical {Gram--Schmidt} (BCGS)}
\label{BCGS}
\bigskip
\begin{algorithmic}[]
\STATE {}
\STATE {This algorithm computes a  Q-R decomposition  $M=QR \in \mathbf{R}^{l \times l}$. }
\STATE {}
\STATE { \textbf{Input:} $M=\left(\begin{array}{cc}M_1, & M_2 \\\end{array}\right)$, \quad 
$M_1=\left(\begin{array}{cc}A  \\B^{T}  \\\end{array}\right)$, \quad 
$M_2=\left(\begin{array}{cc}B  \\-C  \\\end{array}\right).$}
\STATE {}

\STATE {\textbf{Output:} $Q \in \mathbf{R}^{l \times l}$, $Q$ orthogonal. }
\STATE {\textbf{Output:} $R \in \mathbf{R}^{l \times l}$, $R$ upper triangular. }
\STATE {}
\bigskip
\begin{itemize}
\item  {$M_1=Q_1R_1$ \{the thin Householder Q-R\}  } 
\item{$S=Q_1^TM_2$}
\item {$Y=M_2-Q_1S$}
\item {$Y=Q_2R_2$  \{the thin Householder Q-R\}}
\item {$Q=\left(\begin{array}{cc}Q_1, & Q_2 \\\end{array}\right)$, \quad
 $R=\left(\begin{array}{cc}R_1 & S  \\0 & R_2 \\\end{array}\right)$}
\end{itemize}
\end{algorithmic}
\end{algorithm}

Now we present the  Reorthogonalized Block Gram--Schmidt (BCGS2)  method, which is a generalization of the classical {Gram-Schmidt}
 method with reorthogonalization (CGS2),   first analysed by Abdelmalek (see  \cite{Abd71}, \cite{leon}).

\begin{algorithm}[H]
\caption{ Reorthogonalized Block Classical {Gram--Schmidt} (BCGS2)}
\label{BCGS2}
\bigskip
\begin{algorithmic}[]
\STATE {}
\STATE {This algorithm computes a  Q-R decomposition  $M=QR \in \mathbf{R}^{l \times l}$. }
\STATE {}
\STATE { \textbf{Input:} $M=\left(\begin{array}{cc}M_1, & M_2 \\\end{array}\right)$, \quad 
$M_1=\left(\begin{array}{cc}A  \\B^{T}  \\\end{array}\right)$, \quad 
$M_2=\left(\begin{array}{cc}B  \\-C  \\\end{array}\right).$}
\STATE {}

\STATE {\textbf{Output:} $Q \in \mathbf{R}^{l \times l}$, $Q$ orthogonal. }
\STATE {\textbf{Output:} $R \in \mathbf{R}^{l \times l}$, $R$ upper triangular. }
\STATE {}
\bigskip
\begin{itemize}
\item  {$M_1=Q_1R_1$ \{the thin Householder Q-R\}  } 
\item{$S_1=Q_1^TM_2$}
\item {$Y_1=M_2-Q_1S_1$ }
\item {$Y_1=Q_2R_2$  \{the thin Householder Q-R\}}
\item {Reorthogonalization}
\begin{itemize}
\item{$S_2=Q_1^TQ_2$}
\item {$Y_2=Q_2-Q_1S_2$}
\item {$Y_2=Q_2^{(new)}\bar{R_2}$  \{the thin Householder Q-R\}}
\end{itemize}
\item  $S^{(new)}=S_1+S_2 R_2$
\item  $R_2^{(new)}=\bar{R_2} R_2$
\item {$Q=\left(\begin{array}{cc}Q_1, & Q_2^{(new)} \\\end{array}\right)$, \quad
 $R=\left(\begin{array}{cc}R_1 & S^{(new)}  \\0 & R_2^{(new)} \\\end{array}\right)$}
\end{itemize}
\end{algorithmic}
\end{algorithm}

\begin{remark}
Notice that, since  $M_2=Q_1 S_1 +Q_2 R_2$ and $Q_2=Q_1 S_2 +  Q_2^{(new)} \bar{R_2} $, we have  
\[
M_2= Q_1 (S_1 +S_2 R_2) +  Q_2^{(new)} \bar{R_2} R_2 = Q_1 S^{(new)} +  Q_2^{(new)}  R_2^{(new)}.\]

This leads to the Q-R decomposition
\[
M=(M_1, M_2)= (Q_1, Q_2^{(new)}) \left(\begin{array}{cc}R_1 & S^{(new)}  \\0 & R_2^{(new)} \\\end{array}\right).
\]

It is clear that in the  theory these two methods BCGS and BCGS2 are equivalent; however, their numerical properties are different. This is highlighted  in our numerical experiments in Section $4$.
\end{remark}

\section{Error analysis}

We prove that the algorithm BCGS2 (Reorthogonalized Block Classical {Gram-Schmidt}), using the thin Householder Q-R
decomposition,  satisfying (\ref{house1})-(\ref{house2}),    implemented in floating point arithmetic, is backward  stable as a method of solving linear system of equations (\ref{M}) under natural conditions.
 This  means that the computed vector $\tilde z$  is the exact  solution to a slightly perturbed  system $Mz=f$. 
The precise definition is as follows.
\begin{definition}\label{backward stability}
An algorithm for solving nonsingular system of equations  $M z = f$, where $M\in \mathbb R^{l \times l}$, 
and $f \in \mathbb R^{l}$, is  {\bf \emph{backward stable}}, if the computed
result $\tilde z$ in floating point arithmetic  with machine precision  $\macheps$  satisfies
\begin{equation}\label{bstab}
(M+  \Delta M) \tilde z = f + \Delta f, \quad \normt{\Delta M}   \leq  \macheps c_1 \normt{M},
\quad \normt{\Delta f} \leq \macheps c_2  \normt{f},
\end{equation}
where $c_i=c_i(l)$ ($i=1,2$)  are small constants  depending upon $l$.
\end{definition}

\begin{definition}\label{forward stability}
An algorithm for solving nonsingular system of equations $M z = f$, where $M\in \mathbb R^{l \times l}$, 
and $f \in \mathbb R^{l}$, is  {\bf \emph{forward  stable}}, if the computed
result $\tilde z \neq 0$ in floating point arithmetic  satisfies
\begin{equation}\label{fstab}
\frac{\normt{\tilde z- z_{*}}}{\normt{z_{*}}} \leq  \macheps c_3  \kappa(M),
\end{equation}
where $\kappa(M) = \normt{M^{-1}} \normt{M}$ is
the  condition number of $M$, $c_3=c_3(l)$ is a small constant depending
upon $l$, and $z_{*}$ denotes the exact solution to $Mz=f$.
\end{definition}

It is well known  that backward stability implies forward
stability. However, opposite implication is not true (for examples for  problem (\ref{M}), see
Section 4).

We turn now to the issue of stability of algorithms for solving nonsingular system $Mz=f$ using  the Q-R decomposition of the matrix $M$. We assume that if $M=QR$,  then the solution of the linear system of equations $Mz=f$ is obtained  from the triangular system  $Rz=Q^Tf$.   In floating point arithmetic the computed $\tilde Q$ is not exactly orthogonal. How does departure $\tilde Q$ from the orthogonality influence  the computed solution $\tilde z$? 

We begin  with the following useful lemma.

\begin{lemma}\label{Lemat1}
Assume that $\tilde Q \in \mathbb{R}^{l\times l} $  satisfies 
\begin{equation}\label{eqs300}
\normt{I- \tilde Q^T \tilde Q}  \leq  \beta <  1.
\end{equation}

Then $\tilde Q$ is nonsingular and  we have 
\begin{equation}\label{eqs301}
\normt{\tilde Q} \leq \sqrt{1+\beta},
\end{equation}
\begin{equation}\label{eqs302}
\normt{\tilde Q^{-1}} \leq 1/ \sqrt{1-\beta},
\end{equation}
\begin{equation}\label{eqs303}
\normt{I - \tilde Q \tilde Q^T}  \leq \beta.
\end{equation}
\end{lemma}

\begin{proof}
Denote $F=I- \tilde Q^T \tilde Q$. Then we get
\[
\normt{\tilde Q}^2=\normt{\tilde Q^T \tilde Q}= \normt{I-F} \leq \normt{I}+ \normt{F} \leq 1+ \beta,\] 
so (\ref{eqs301}) holds. Since $\normt{F} <1$  we conclude that $I-F $ is nonsingular. Thus, the matrix $\tilde Q$  also is nonsingular. 

Notice that $\normt{\tilde Q^{-1}}^2= \normt{(I-F)^{-1}} \leq 1/(1- \normt{F}) \leq 1/(1-\beta)$, and    (\ref{eqs302}) is proved. 

We observe that $I - \tilde Q \tilde Q^T = \tilde Q^{-T} (I- \tilde Q^T \tilde Q)\tilde Q^{T}$. We see that 
\[
\normt{I - \tilde Q \tilde Q^T}= \rho ({I - \tilde Q \tilde Q^T})= \rho({I - \tilde Q^T \tilde Q})= \normt{I- \tilde Q^T \tilde Q},
\]
where $\rho(.)$ denotes the spectral radius. This together with (\ref{eqs300}) completes the proof of (\ref{eqs303}). 
\end{proof}

\begin{theorem}\label{tw1}
Let $M\in \mathbb R^{l \times l}$ be nonsingular and suppose that   $\tilde  Q\in \mathbb R^{l \times l}$ and  
 $\tilde R \in  \mathbb R^{l \times l}$ satisfy
\begin{equation}\label{eqs10a}
\normt{ M- \tilde Q  \tilde  R }  \leq  \alpha  \normt{M }, \quad \alpha \kappa(M) < 1
\end{equation}
and
\begin{equation}
\normt{ I  -\normal{\tilde Q} }  \leq  \beta <1. \label{eqs10b}
\end{equation}

Moreover, assume that there exist $\Delta R$ and $\Delta Q$ such that 
\begin{equation}\label{eqs11}
(\tilde R + \Delta R) \tilde z = (\tilde Q + \Delta Q)^T f, 
\end{equation}
where 
\begin{equation}\label{eqs12}
\normt{\Delta R} \leq \gamma \normt{\tilde R}, \quad \normt{\Delta Q} \leq \delta \normt{\tilde Q}.
\end{equation}

Then we have 
\begin{equation}\label{eqs13}
(M+\Delta M) \tilde z = f+ \Delta f,\quad \normt{\Delta M} \leq \mu \normt{M},  \quad \normt{\Delta f} \leq \nu \normt{f},
\end{equation}
 where  
\begin{equation}\label{eqs14}
\mu= \alpha+\gamma (1+ \alpha) \sqrt{(1+\beta)/(1-\beta)}, \quad \nu= \beta + \gamma (1+\beta).
\end{equation}
\end{theorem}

\begin{proof}
Let   $E=M- \tilde Q \tilde R$ and $F=I- \tilde Q \tilde Q^T$. Since $\alpha \kappa(M) <1$ we see that $M-E= \tilde Q \tilde R$ is nonsigular. Thus, $\tilde Q$ and $\tilde R$ also are nonsingular. 
Multiplying  (\ref{eqs11}) by $\tilde Q$ gives  the identity $\tilde Q  (\tilde R + \Delta R) \tilde z = \tilde Q (\tilde Q + \Delta Q)^T f$, 
which we rewrite as follows
\[
(M-E+ \tilde Q \Delta R) \tilde z = f+ ((\tilde Q \tilde Q^T-I) + \tilde Q \Delta Q^T) f.
\]

Let  us introduce  
\begin{equation}\label{eqs15}
\Delta M=   \tilde Q \Delta R - E, \quad  \Delta f= ((\tilde Q \tilde Q^T-I) + \tilde Q \Delta Q^T) f.
\end{equation}

Then $(M+ \Delta M) \tilde z = f+ \Delta f$. 
 It remains to prove that $\Delta M$ and $\Delta f$ satisfy  (\ref{eqs13})-(\ref{eqs14}).  

Taking norms in (\ref{eqs15}), we get
\begin{equation}
\normt{\Delta M} \leq  \normt{E}+ \normt{\tilde Q} \normt{\Delta R}, \quad  \normt{\Delta f} \leq  \normt{I- \tilde Q \tilde Q^T}+ \normt{\Delta Q} \normt{\tilde Q}.
\end{equation}
This together with    (\ref{eqs10a}) and (\ref{eqs12}) gives

\begin{equation}\label{eqs16}
\normt{\Delta M} \leq \alpha  \normt{M}+ \gamma \normt{\tilde Q} \normt{\tilde R}, \quad  \normt{\Delta f} \leq  \normt{I- \tilde Q \tilde Q^T}+ \delta {\normt{\tilde Q}}^2.  
\end{equation}    

Notice that $M-E= \tilde Q \tilde R$, hence $\tilde R= {\tilde Q}^{-1} (M-E)$. Taking norms, we obtain 
\[
\normt{\tilde R} \leq \normt{\tilde Q^{-1}} \normt{M -E} \leq (1+\alpha) \normt{\tilde Q^{-1}} \normt{M}.
\]

This together with  (\ref{eqs16}) and Lemma \ref{Lemat1} gives  (\ref{eqs13})-(\ref{eqs14}).  This  completes the proof.
\end{proof}

Now we apply the results on the numerical properties of the Q-R factorization. 
The detailed error analysis of Algorithms $1$-$2$ was given in \cite{AlaJesse}. 
In particular,  J~.L.~Barlow and A.~Smoktunowicz proved  the following theorem.

\begin{theorem}\label{tw2} 
Let  $M=\left(\begin{array}{cc}M_1, & M_2 \\\end{array}\right)$, where $M_1 \in   \mathbb R^{(m+n) \times m}$,
$M_2 \in   \mathbb R^{(m+n) \times n}$.    
Assume that $M$ is nonsingular.
Then  Algorithm $1$ (BCGS), implemented in floating point arithmetic with machine precision  $\macheps$,
  produces computed $\bar  Q= (\bar Q_1, \bar Q_2)$ and 
$\bar R =  \left(\begin{array}{cc} \bar R_1 & \bar S  \\0 & \bar R_2 \end{array}\right) $ that satisfy  
\begin{equation}
\normt{ M-\bar Q \bar  R }  \leq  \macheps L_1(m,n)  \normt{M }, \\  \label{eqs1}
\end{equation}
and
\begin{equation}
\normt{ I  -\normal{\bar  Q} }  \leq  \macheps L_2(m,n)  \, \normt{M_2} \normt{\bar R_2^{-1}}. \label{eqs2}
\end{equation}

 Algorithm $2$ (BCGS2)  produces computed $\tilde Q$ and $\tilde R$ such that 
\begin{equation}
\normt{ M-\tilde Q \tilde R } \leq  \macheps L_3(m,n)  \normt{M }.   \label{eqs3}
\end{equation}

Moreover, if 
\begin{equation}
\macheps L_4(m,n)  \, \normt{M_2} \normt{\bar R_2^{-1}} < 1, \label{eqs4}
\end{equation}
then  
\begin{equation}
\normt{ I  -\normal{\tilde Q} } \leq \macheps L_5(m,n),  \label{eqs5}
\end{equation}
where $L_i(\cdot)$ ($i=1, \ldots, 5$)  are modestly  growing function on  $m$ and $n$.
\end{theorem}

\begin{remark}
Notice that  $ \normt{M_2} \leq \normt{M}$ and $ \normt{\bar R_2^{-1}} \leq   \normt{\bar R^{-1}}$, so 
  (\ref{eqs4}) holds  for stronger assumption   that   $\macheps L_6(m,n)   \kappa(M) < 1$, where  $L_6(m,n)$ is modestly growing  function on  $m$ and $n$.  This assumption is both natural and typical for this context.
\end{remark}

\begin{remark}
Applying  the results for the standard methods for solving triangular systems  of equations (see \cite{Golub96}) , we see  that the quantities $\alpha$, $\delta$ and $\gamma$ defined in Theorem \ref{tw1}  are at  the level machine precision $\macheps$.  This together with Theorems \ref{tw1} and \ref{tw2} implies that Algorithm $2$ (BCGS2) is {\bf backward stable}, i.e. (\ref{bstab}) holds. 
However, our numerical tests in Section $4$ show that Algorithm $1$ (BCGS)  is unstable, in general. 
\end{remark}

\section{Numerical experiments}
We present a  comparison of Algorithms $1$ and  $2$. Numerical tests were performed  in \textsl{MATLAB},  with machine precision $\macheps \approx 2.2 \cdot 10^{-16}$.

In our tests we chose $z_{*}$, formed   $f=M*z_{*}$, computed  a Q-R factorization  $M=QR$ using Algorithm $1$ (Block Classical {Gram-Schmidt}- BCGS) and 
Algorithm $2$ (Reorthogonalized Block Classical {Gram-Schmidt}- BCGS2) respectively,
and solved the system of equations $Mz=f$ from the triangular system of equations $Rz=Q^Tf$, using the \textsl{MATLAB}  command: \begin{verbatim} z=R\(Q'*f). \end{verbatim}

We report the following statistics for each algorithm:

\begin{itemize}
\item ${orth}_{Algorithm}= \frac{\|I-\tilde Q^T \tilde Q\|}{\macheps}$ (the orthogonality error),

\item ${dec}_{Algorithm}= \frac{\|M-\tilde Q* \tilde R\|}{\macheps \|M\|}$ (the  decomposition error),

\item ${res}_{Algorithm}= \frac{\|M \tilde z - f \|}{\macheps \|M\| \, \|\tilde z\|}$ (the   backward stability error),

\item ${stab}_{Algorithm}= \frac{\|\tilde z - z_{*} \|}{\macheps \kappa(M) \, \|\tilde z\|}$ (the  forward stability error).
\end{itemize}

\medskip

Note that, in the theory,  ${stab}_{Algorithm} \leq {res}_{Algorithm}$, that is,  the backward stability implies the forward stability.

We consider  several test matrices. 
The function  $\tt {matrix1}$ returns  the matrix $X(m \times n)$, where $m \ge n$,  from  the singular value
 decomposition $X = P D Q$, where $P(m \times  n)$ is left orthogonal ($P^TP=I$)  and
 $Q(n \times n)$ is an  orthogonal matrix. They are generated by the $\tt  {orth}$  subroutine  in  \textsl{MATLAB}.
Here $D(n \times n)$ is a diagonal matrix,  computed as $D=diag(logspace(0,-s,n))$. 
 The   \textsl{MATLAB} command $\tt{logspace(0,-s,n)}$ generates $n$ points between decades $1$ and $10^{-s}$.
Then the condition number of $X$ will be approximately  equal to $\kappa(X)= 10^s$.    

The function  $\tt {matrix2}$ returns  the symmetric positive definite matrix  $X(n \times n)$, 
from the spectral decomposition $X=PDP$, where $P(n \times n)$ is orthogonal and $D$ is diagonal, generated 
in the same way as in the function $\tt {matrix1}$.

 The \textsl{MATLAB} codes for these functions are as follows.     
\begin{verbatim}
function X=matrix1(m,n,s);
% X(mxn), cond(X)=10^s.
D=diag(logspace(0,-s,n));
randn('state',0);
P=orth(randn(m)); P=P(:,1:n);
Q=orth(randn(n)); X=P*D*Q';
end
\end{verbatim}
\begin{verbatim}
function X=matrix2(n,s);
% X(nxn),  cond(A)=10^s.
% X is symmetric positive definite. 
D=diag(logspace(0,-s,n));
randn('state',0);
P=orth(randn(n));
A=P*D*P'; A=(A+A')/2;
end
\end{verbatim}

The \textsl{MATLAB} command \verb+randn('state',0)+ is used
to reset the random number generator to its initial state.

In all of our tests we use the scaling to change the condition number of the  matrix $M$ of (\ref{M}).  
More precisely, for given parameter  $0 \neq t \in \mathbb{R}$ and given initial matrices $A_1$, $B_1$ and $C_1$,  we form new matrices as follows:
\begin{equation}\label{At}
A=A_1/t, \quad C=C_1*t,  \quad B=B_1*t.
\end{equation}
Note that the condition numbers  $\kappa(A)$, $\kappa(B)$  and $\kappa(C)$ remain unchanged,  but the condition number of the matrix $M$ of (\ref{M}) can be very large
for a particular $t$.
Then we compute $f=M*z_{*}$, where 
\begin{equation}\label{zt}
z_{*}=(x_{*}^T, y_{*}^T)^T, x_{*}= t(1,1,\ldots,1)^T \in \mathbf{R}^m,  y_{*}=\frac{1}{t} (1,1,\ldots,1)^T \in \mathbf{R}^n.
\end{equation}

\medskip

\begin{example}\label{example1}
Let $A_1=H$, where $H$ is the Hilbert matrix: 
\[
h_{i,j}=\frac{1}{i+j-1},\,\, i, j = 1, \ldots, m.
\]

We take $B_1= \tt{matrix1}(m,n,sB)$, with $m=12$, $n=6$, $sB=10$,  and $C_1=e e^T$, where $e=(1,1,\ldots,1)^T \in \mathbf{R}^n$ and then we form the matrices $A$, $B$, $C$ and the vectors $z_{*}$ and $f$, according to (\ref{At})-(\ref{zt}).

Table $1$  reports all statistics for Algorithms $1$ and $2$.
We observe that in the above example $A$ is very ill-conditioned. We see that Algorithm $1$ produces a backward stable solution only for $t=0.01$, otherwise Algorithm 1 (BCGS) is unstable.
It is evident that Algorithm $2$ is perfectly backward stable.

\begin{table}\label{tabelka1}
\caption{Results for Example $1$ for the computed solutions to $Mz=f$ for $m=12$, $n=6$.
Here $\kappa(A)= 1.67 \cdot 10^{16}$,  $\kappa(B) =  10^{8}$.}

\medskip

\begin{center}
\begin{tabular}{|c|c|c|c|c|c|c|c|}
\hline
$t$ &   $0.01$ & $0.1$  & $1$  & $10$ & $100$ \\
\hline $\kappa(M)$ 
&     2.1862e+14&   2.2212e+10&   2.4123e+08 &   2.1104e+10 &   2.1103e+12\\
 \hline $orth_{BCGS}$
 &     1.1085e+10 &   1.3448e+08 &   2.3227e+07 &   2.2281e+07 &   2.8924e+07\\
\hline $orth_{BCGS2}$
&     3.9722&    6.2250&    4.3518&    3.3968&    4.3044\\
\hline $dec_{BCGS}$
&     0.7918&    1.3793&    0.6659&    0.4672&    0.3931\\
\hline $dec_{BCGS2}$ 
 &      0.7918&    1.3793&    0.6873&    0.6873&    0.7554 \\
\hline $res_{BCGS}$  
&      20.6000&   2.0240e+05&   6.3450e+06&   2.5483e+04&   2.1652e+03\\
\hline $res_{BCGS2}$ 
&     6.0151e-04&    0.0654&    1.0473&    0.0274&    0.0371 \\
\hline $stab_{BCGS}$ 
 &    20.6000&   2.0239e+05&   5.4530e+06&   2.2373e+04&   2.0754e+03\\
\hline $stab_{BCGS2}$  
&     6.0608e-05&    0.0113&    0.1755&   6.9555e-04&   3.7342e-05\\
\hline
\end{tabular}
\end{center}
\end{table}
\end{example}

\begin{example}\label{example2}
 We take $A_1=\tt{matrix2}(m,sA)$, $B_1=\tt{matrix1}(m,n,sB)$  and $C_1=\tt{matrix2}(n,sC)$ for large values of $m$ and $n$. Here $sA=sB=sC=10$.
The results are shown in Tables $2$ and $3$.
We see that a more severe instability of Algorithm $1$ than in Example \ref{example1}. Clearly  it depends on  the orthogonality of BCGS.
Algorithm $2$ gives very satisfactory results.
\medskip

\begin{table}
\caption{Results for Example $2$ for the computed solutions to $Mz=f$ for $m=1000$, $n=500$.
Here $\kappa(A)=  \kappa(B) = \kappa(C)= 10^{10}$.} 

\medskip

\begin{center}
\begin{tabular}{|c| c| c| c| c| c| c|c|}
\hline
$t$ &   $0.01$ & $0.1$  & $1$  & $10$ & $100$ \\
\hline $\kappa(M)$ 
&   4.3095e+12&   1.0001e+10&   1.6088e+10&   1.1008e+12&   1.1034e+14 \\
 \hline $orth_{BCGS}$
 &     2.8711e+08&   5.6834e+07&   8.7330e+06&   1.4622e+06&   8.4625e+05 \\
\hline $orth_{BCGS2}$
&   33.4687&   42.9708&   30.4565&   24.1932&   22.9869 \\ 
\hline $dec_{BCGS}$
&    9.8999&   11.6189&    6.7868&    9.0066&   11.2679 \\
\hline $dec_{BCGS2}$ 
 &      9.8999&   11.6189&    6.8005&    9.6463&   12.0870 \\
\hline $res_{BCGS}$  
&     881.9107&   1.9698e+04&   9.5569e+04&   1.2498e+04&   1.1409e+04\\ 
\hline $res_{BCGS2}$ 
&    3.0624e-04&    0.0328&    1.2607&    1.0745&    0.9646 \\
\hline $stab_{BCGS}$ 
 &   751.9935&   1.0728e+04&   79.6715&    0.1522&    5.4143 \\
\hline $stab_{BCGS2}$  
&      1.1538e-05&    0.0037&    0.1044&    0.0317&    0.0332\\
\hline
\end{tabular}
\end{center}
\end{table}

\medskip

\begin{table}
\caption{Results for Example $3$ for the computed solutions to $Mz=f$ for $m=3000$, $n=100$.
Here $\kappa(A)=  \kappa(B)= \kappa(C)= 10^{10}$.} 

\medskip

\begin{center}
\begin{tabular}{|c| c| c| c| c| c| c|c|}
\hline
$t$ &   $0.01$ & $0.1$  & $1$  & $10$ & $100$ \\
\hline $\kappa(M)$ 
&   8.6976e+12&   1.8422e+10&   1.6098e+10&   1.0625e+12&   1.0652e+14\\
 \hline $orth_{BCGS}$
 &       5.4204e+08&   1.1497e+08&   2.4793e+07&   6.2835e+06&   1.8672e+06\\
\hline $orth_{BCGS2}$
&    41.3322&   37.9195&   35.8591&   31.9450&   36.1455\\
\hline $dec_{BCGS}$
&    10.5139&   12.7317&    9.3857&    8.7265&    7.4901\\
\hline $dec_{BCGS2}$ 
 &   10.5139&   12.7317&    9.3861&    8.4227&    7.6466\\
\hline $res_{BCGS}$  
&     630.5847&   3.0521e+04&   1.2111e+05&   1.0972e+04&   5.3586e+03\\ 
\hline $res_{BCGS2}$ 
&     0.0011&    0.1154&    1.3523&    0.4655&    0.4709 \\  
\hline $stab_{BCGS}$ 
 &   293.7614&   2.4108e+04&  443.6943&    0.1661&    6.1348\\
\hline $stab_{BCGS2}$  
&      3.0067e-06&    0.0073&    0.1495&    0.0118&    0.0138\\
\hline
\end{tabular}
\end{center}
\end{table}
\end{example}

\medskip

\section{Conclusions}
We  proposed  and analyzed  two  algorithms  for solving symmetric saddle point problems. 
Although the {Block Classical {Gram--Schmidt} (BCGS is unstable),  the Reorthogonalized Block Classical {Gram--Schmidt} (BCGS2) gives acceptable results in floating point arithmetic.
Algorithm BCGS2 is suitable for parallel implementation.  In order to refine   the numerical solutions of linear systems, one can apply 
 some iterative refinement techniques  (see \cite{smok3} ).

\medskip

\noindent {\bf Acknowledgements.} The authors are  grateful to  Mike Peter West
for the many suggestions, which helped to improve the paper.

\end{document}